\documentclass{article}
\usepackage{amsmath,amssymb,amsfonts,amsthm,epsfig,enumerate}

\newtheorem{thm}{Theorem}
\newtheorem{prop}[thm]{Proposition}

\newtheorem{prob}[thm]{Problem}
\newtheorem{lem}[thm]{Lemma}

\newtheorem{example}[thm]{Example}

\newcommand{\IN}{\ensuremath{{\mathbb N}}}

\newcommand{\IZ}{\ensuremath{{\mathbb Z}}}

\newcommand{\chia}{\ensuremath{\chi_{g_A}}}
\newcommand{\chib}{\ensuremath{\chi_{g_B}}}

\newcommand{\gcn}[4]{\ensuremath{{}^{({#1},{#2})}\chi_{g#4}({#3})}}
\newcommand{\gcnt}[4]{\ensuremath{{}^{({#1},{#2})}\chi_{g#4}^{\Theta}({#3})}}
\newcommand{\gcna}[3]{\ensuremath{\gcn{#1}{#2}{#3}{_A}}}
\newcommand{\gcnb}[3]{\ensuremath{\gcn{#1}{#2}{#3}{_B}}}

\title{Colouring games based on autotopisms of Latin hyper-rectangles}
\author{Stephan Dominique Andres\\
{\small FernUniversit\"{a}t in Hagen, Fakult\"{a}t f\"{u}r Mathematik und Informatik}\\
{\small IZ, Universit\"{a}tsstr.\ 1, 58084 Hagen, Germany}\\
{\small \texttt{dominique.andres@fernuni-hagen.de}}\\
Ra\'ul M.\ Falc\'on\\
{\small  University of Seville, School of Building Engineering}\\
{\small  YOUR ADDRESS, Seville, Spain}\\
{\small  \texttt{rafalgan@us.es}}\\}

\begin{document}
\maketitle


\begin{abstract}
Every partial colouring of a Hamming graph is uniquely related to a partial Latin hyper-rectangle. In this paper we introduce the $\Theta$-stabilized $(a,b)$-colouring game for Hamming graphs, a variant of the $(a,b)$-colouring game so that each move must respect a given autotopism $\Theta$ of the resulting partial Latin hyper-rectangle. We examine the complexity of this variant by means of its chromatic number. We focus in particular on the bi-dimensional case, for which the game is played on the Cartesian product of two complete graphs, and also on the hypercube case.
\end{abstract}

\section{Introduction}

The {\em colouring game} dates back to an idea of Brams, which was published by Gardner~\cite{gardner} in 1981, and was popularised by Bodlaender \cite{Bodlaender1991} in 1991. Based on the graph colouring problem, this 
game is played on a finite graph by two players, Alice ($A$) and Bob ($B$), with Alice 
playing first. They must alternately colour some uncoloured vertex of the graph with a colour taken from a given set so  that none two adjacent vertices are coloured with the same colour. A move in the game consists, 
therefore, in colouring exactly one vertex at a time. If all vertices of the graph are coloured at the end of the game,  then Alice wins, otherwise Bob wins. Bodlaender dealt with the complexity of determining if there 
exists a winning strategy for one of the players. In this regard, he introduced the {\em game chromatic number} $\chi_g(G)$ as the least integer $k$ such that Alice has 
a winning strategy when the game is played on a graph $G$ by using $k$ colours. 
Since Alice wins in any case whenever the game is played with $n$ colours on an $n$-vertex graph, the game chromatic number is a well-defined integer.
During the last decades many efforts using different methods from graph theory have been done to reduce the upper bound for the 
game chromatic number of planar graphs, cf.\ \cite{bartnickietal,kierstead,zhurefined}.

As the game may change significantly when Bob begins instead of Alice, later different authors \cite{andresforest,zhucartesian} distinguish between the game chromatic numbers $\chi_{g_A}(G)$ resp.\ $\chi_{g_B}(G)$
for the game where Alice begins resp.\ where Bob begins, the above notation was first used by Andres~\cite{andresgpg}. 

As a generalization of the colouring game, Kierstead \cite{Kierstead2005} introduced the {\em $(a,b)$-colouring game}, which assumes the rule that moves of Alice and Bob consist in 
colouring, respectively, $\min\{a,u\}$ and $\min\{b,u\}$ 
distinct uncoloured vertices, where $u$ denotes the number of uncoloured vertices before the move. We denote the $(a,b)$-colouring game by $g_A$ resp.\ $g_B$ depending on the rule 
whether Alice resp.\ Bob has to perform the first move.
If \mbox{$a=b=1$}, then $g_A$ is just the colouring game.  
The {\em $(a,b)$-game chromatic numbers} $\gcna{a}{b}{G}$ and $\gcnb{a}{b}{G}$ are then defined as the least integer $k$ such that 
Alice has a winning strategy when the respective $(a,b)$-colouring game  is played on a graph $G$ by using $k$ colours. 
In 2009, Andres 
\cite{Andres2009} generalized this new game to digraphs. Shortly after,  Schlund \cite{Schlund2011} focused on 
partial Latin squares of a given order. Recall that a {\em partial Latin square of order} $n$ 
is an $n\times n$ array in which each cell is either empty or contains one  element chosen from a set of $n$ symbols, such that each symbol occurs at most once in each row and in each column. This is a {\em Latin square} 
if there are no empty cells in such an array. 
Schlund introduced the digraph  whose vertices are all possible partial Latin squares of order $n$ and where, given two such partial Latin squares, $P$ and $P'$, there exists 
a directed edge from $P$ to $P'$ if and only if $P$ is a subsquare of $P'$ and  $P'$ has exactly one more non-empty cell than $P$. 
He focused in particular on determining lower and upper bounds for the chromatic number 
of partial latin squares. In a more general way, it was Bose \cite{Bose1963} who introduced  the study of graphs related to Latin squares. In a recent paper, Besharati et al. \cite{Besharati2016} have studied the 
chromatic number 
of these graphs in case of dealing with Latin squares with a certain symmetric structure.  They have focused on the study of row-complete and circulant Latin squares.
Schlund also was the first one who considers the game chromatic number of latin squares, which is in the language of graph theory simply the game chromatic number of its rook's graphs, cf.\ Fig.~\ref{bildrooks}.

\begin{figure}[htbp]
\begin{center}
\includegraphics[scale=0.5]{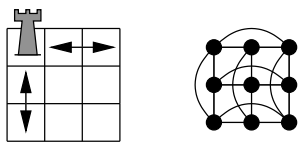}
\end{center}
\caption{\label{bildrooks}A latin square and its rook's graph}
\end{figure}

This paper deals with a natural generalization of Schlund's results to partial Latin hyper-rectangles having a given symmetry in their autotopism group. The structure of the paper is 
the following. In Section~\ref{secprelim} we expose some preliminary concepts and results on graphs and partial Latin hyper-rectangles that we use throughout the paper. The graph colouring game on the latter with regard to an autotopism is introduced in Section~\ref{secgame}. We focus in particular on the bi-dimensional case. Finally, in Section~\ref{secmodifiedgame} we study a modified game based on principal isotopisms which, by a central concept, the Orbit Contraction Lemma, is equivalent to the aforementioned game and leads in some important cases to a simplification of the analysis of the used strategies.

\section{Preliminaries}\label{secprelim}

In  this section we introduce some basic concepts, notations and results on graphs and partial Latin hyper-rectangles that are used throughout the paper. For more details about these topics we refer, respectively, to the 
monographs of Harary~\cite{Harary1969} resp.\ Diestel \cite{diestel} and D\'enes and Keedwell \cite{Denes1991}.

\pagebreak[3]

\subsection{Graph Theory}

 A {\em graph} is a pair $G=(V,E)$ formed by a set $V$ of {\em vertices} and a set $E$ of {\em edges} that contain two vertices. This is {\em vertex-weighted} if each one of its 
vertices has assigned a numerical value or 
{\em weight}. 
 The number $|V|$ of vertices of $G$ is its {\em order}. Two vertices that are contained in the same edge are said to be {\em adjacent}. This edge is then said to be {\em incident} to both vertices. The {\em degree} of a 
 vertex is the number of edges that are incident to such a vertex. The maximum vertex degree of the graph $G$ is denoted as $\Delta(G)$. A graph is said to be $k$-{\em regular} if all its vertices have the same degree 
$k$.  If any two vertices of $G$ are adjacent, then the graph is said to be {\em complete}. The complete graph of $n$ vertices is denoted as~$K_n$. The {\em contraction} of a pair of vertices of $G$ gives rise to a new 
graph  where both vertices and their incident edges are eliminated and replaced by a single vertex that is adjacent to all those vertices that were adjacent to the former.

A {\em $k$-partial vertex labeling} of $G$  is any map that assigns a set of $k$ labels to a subset of vertices of $V$. A {\em partial $k$-colouring} of $G$ is a {\em $k$-partial 
vertex labeling} of the graph 
with the property that
none two adjacent vertices have the same label. 
The labels are also called \emph{colours}.
If none vertex is uncoloured, then a partial $k$-colouring is called a {\em 
 $k$-colouring} of the graph. The smallest number of colours that are required to determine one such a colouring of $G$ is its {\em chromatic number} $\chi(G)$. In particular, $\chi(G)\leq \Delta(G)+1$, for any graph 
 $G$. 
The problem of deciding whether the chromatic number of a graph is at most $k$ is NP-hard for $k\ge3$.
This problem is known as the {\em graph colouring problem}.

\pagebreak[4]

 The {\em Cartesian product} of two graphs $G_1=(V_1,E_1)$ and $G_2=(V_2,E_2)$ is the graph $G_1\Box G_2$ whose set of vertices coincides with the Cartesian product $V_1 \times V_2$ and where two distinct vertices 
 $(u_1,u_2)$ and $(v_1,v_2)$ are adjacent if and only if $u_i=v_i$ and $u_j$ is adjacent to $v_j$ in $G_j$ for some $i,j\in\{1,2\}$ with $i\neq j$. If, besides, there exists an 
edge whenever $u_i$ is adjacent to $v_i$ 
 in $G_i$, for all $i\leq 2$, then this constitutes the {\em strong product} $G_1\boxtimes G_2$ of $G_1$ and $G_2$. The {\em Hamming graph} is defined as the Cartesian product $\mathcal{H}_{n_1,\ldots,n_d}:=K_{n_1}\Box 
\ldots\Box K_{n_d}$. This is an $(n_1+\ldots+n_d-d)$-regular graph of order $\prod_{j=1}^dn_j$. The Hamming graph $\mathcal{H}_{3,3}$ is the rook's graph depicted in Figure~\ref{bildrooks}. For $d=2$, 
Schlund~\cite{Schlund2011} proved by a simple simulation argument (see loc. cit. Lemma 4.9) that
\[\gcnb{2a+1}{1}{\mathcal{H}_{n,n}}\le\gcnb{a}{1}{\mathcal{H}_{n,n}}.\]

The next lemma generalizes this result to arbitrary graphs and parameters.

\begin{lem}\label{lem0} Let $G$ be a finite graph, $X\in\{A,B\}$ and $a$, $b$ and $m$ be three positive integers greater than or equal to $1$. Then,
\[{}^{(ma+(m-1)b,b)}\chi_{g_X}(G)\leq {}^{(a,b)}\chi_{g_X}(G)\]
\end{lem}

\begin{proof} Assume that Alice has a winning strategy for the $(a,b)$-colouring game on $G$. She uses this strategy for the
$(ma+(m-1)b,b)$-game on $G$. When she has to move on Bob's turns, she simulates Bob's move by an arbitrary move.
The number $ma+(m-1)b$ of her turns guarantees that Bob will not have any of her moves from the $(a,b)$-game in the $(ma+(m-1)b,b)$-game.
\end{proof}

\vspace{0.5cm}

\subsection{Partial Latin hyper-rectangles}

Let $d\geq 2$  be a positive integer. By a \emph{line} in an $(n_1\times\ldots\times n_d)$-array we mean the set of cells that is obtained if we fix each coordinate except for one.
An {\em $n_1\times\ldots\times n_d$ partial Latin hyper-rectangle} based on the set $[n]=\{1,\ldots,n\}$ is an $(n_1\times\ldots\times n_d)$-array that satisfies the so-called {\em 
Latin 
array condition}:  each cell is either empty or contains one symbol chosen from the set $[n]$ in such a way that each symbol occurs at most once in each line of the array. 
Its {\em dimension} is $d$. If $d=2$, then this 
corresponds to a {\em partial Latin rectangle} (a {\em partial Latin square} if $n_1=n_2=n$). For higher orders, if $n_1=\ldots=n_d$, then this corresponds to a {\em partial Latin hypercube}. If the array does not contain empty cells, then the adjective partial is eliminated in each one of the previous definitions. From here on, $\mathcal{PLH}_{n_1,\ldots,n_d,n}$ denotes the set of $n_1\times\ldots\times n_d$ partial Latin hyper-rectangles based on $[n]$. Figure \ref{Fig1} shows three partial Latin rectangles in the set $\mathcal{PLH}_{3,4,4}$.

\begin{figure}[htbp]
\begin{center}
$\begin{array}{ccccc}P_1\equiv\begin{array}{|c|c|c|c|} \hline
1 & \ & 4 & \ \\ \hline
\ & 2 & \ & \ \\ \hline
2 & \ & \ & 1 \\ \hline
\end{array} & \ & P_2\equiv\begin{array}{|c|c|c|c|} \hline
\ & \ & 2 & \ \\ \hline
\ & 1 & \ & 3 \\ \hline
1 & 2 & \ & \ \\ \hline
\end{array} & \ & P_3\equiv\begin{array}{|c|c|c|c|} \hline
1 & \ & 3 & \ \\ \hline
\ & 2 & \ & 4\\ \hline
2 & 3 & 4 & 1 \\ \hline
\end{array}\end{array}$
\caption{Partial Latin rectangles in the set $\mathcal{PLH}_{3,4,4}$.}
\label{Fig1}
\end{center}
\end{figure}

The set $\mathcal{PLH}_{n_1,\ldots,n_d,n}$ is uniquely identified with the set of partial $n$-colour\-ings of a vertex-labeled Hamming graph $\mathcal{H}_{n_1,\ldots,n_d}$. To see it, 
observe that every cell of a partial Latin hyper-rectangle $P\in \mathcal{PLH}_{n_1,\ldots,n_d,n}$ is uniquely identified with a tuple $(i_1,\ldots,i_d)\in [n_1]\times \ldots \times [n_d]$, where each $i_j$ represents the position of the cell under consideration in the $j^{\mathrm{th}}$ line of $P$. These tuples can be considered as the labels of the vertices of $\mathcal{H}_{n_1,\ldots,n_d}$ by taking into account that two such vertices are adjacent if and only if their corresponding labels in $[n_1]\times \ldots \times [n_d]$ differ exactly in one component. Each label indicates, therefore, the position in which is situated the cell of $P$ that is uniquely identified with the corresponding vertex of the Hamming graph. This cell is empty if and only if the mentioned vertex is uncoloured. Otherwise, the cell contains a symbol of the set $[n]$ that is identified with the corresponding colour of the vertex. Hence, colouring an uncoloured vertex in a Hamming graph is equivalent to fill with a symbol an empty cell in a partial Latin hyper-rectangle. We say in this case that the cell is {\em coloured} with that symbol. Figure \ref{Fig_Hamming} shows, for instance, the $4$-partial colouring of the labeled Hamming graph related to the partial Latin rectangle $P_1$ of Figure \ref{Fig1}. We have used the style $\bullet$ to represent uncoloured vertices and the styles $\blacktriangle$, $\blacklozenge$ and $\blacktriangledown$ to represent, respectively, those coloured vertices related to the symbols $1$, $2$ and $4$.

\begin{figure}[htbp]
\begin{center}
\includegraphics[width=4cm]{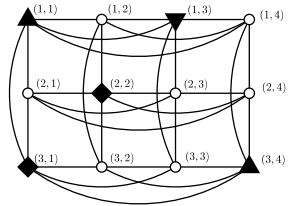}
\end{center}
\caption{Partial colouring of a labeled Hamming graph.}
\label{Fig_Hamming}
\end{figure}

 Every partial Latin hyper-rectangle $P\in \mathcal{PLH}_{n_1,\ldots,n_d,n}$ is represented by its {\em set of entries} $E(P)$, where an {\em entry} of $P$ is any $(d+1)$-tuple 
\[(i_1,\ldots,i_d,P(i_1,\ldots,i_d))\in [n_1]\times \ldots \times [n_d]\times [n].\] 
Here, $P(i_1,\ldots,i_d)$ denotes the symbol that appears in a given non-empty cell $(i_1,\ldots,i_d)$. If the set $E(P)$ is empty, then 
$P$ is called {\em trivial}. Further, if 
 $E(P)\subseteq E(Q)$, for some partial Latin hyper-rectangle $Q\in \mathcal{PLH}_{n_1,\ldots,n_d,n}$, then it is said that $P$ is {\em contained} in $Q$.

  Permutations of lines and symbols of $P$ give rise to new partial Latin hyper-rectangles in $\mathcal{PLH}_{n_1,\ldots,n_d,n}$ that are said to be {\em isotopic} to $P$. In this regard, let $S_m$ and 
 $\mathfrak{I}_{n_1,\ldots,n_d,n}$ respectively denote the symmetric group in $m$ elements and the direct product $S_{n_1}\times\ldots\times S_{n_d}\times S_n$. The isotopic partial Latin hyper-rectangle of $P$ according 
 to an {\em isotopism} $\Theta=(\pi_1,\ldots,\pi_d,\pi)\in \mathfrak{I}_{n_1,\ldots,n_d,n}$ is then denoted by $P^{\Theta}$ and is defined by
\begin{eqnarray*}
&&E(P^{\Theta})\\
&=&\{(\pi_1(i_1),\ldots,\pi_d(i_d),\pi(P(i_1,\ldots,i_d)))\mid\,(i_1,\ldots,i_d,P(i_1,\ldots,i_d))\in E(P)\}.
\end{eqnarray*} 

Hence, 
\[P^{\Theta}(\pi_1(i_1),\ldots,\pi_d(i_d))=\pi(P(i_1,\ldots,i_d)),\] 
for all $(i_1,\ldots,i_d,P(i_1,\ldots,i_d))\in E(P)$. If $\pi$ is the trivial permutation, 
 that is, if $\pi=\mathrm{Id}$, then the isotopism $\Theta$ is called {\em principal}. If \mbox{$P^{\Theta}=P$,} then the isotopism $\Theta$ is said to be an {\em autotopism} of $P$. The set of autotopisms of $P$ is 
 endowed of group structure with the componentwise composition of permutations. The set of non-trivial partial Latin hyper-rectangles having a given isotopism $\Theta\in\mathfrak{I}_{n_1,\ldots,n_d,n}$ in their 
 autotopism group is denoted as $\mathcal{PLH}_{\Theta}$. Observe, for instance, that the triple $((12)(3),(1234),(1)(2)(34))\in\mathfrak{I}_{3,4,4}$ is an isotopism between the partial Latin rectangles $P_1$ and $P_2$ 
 in Figure~\ref{Fig1}. Besides, $P_3\in\mathcal{PLH}_{((12)(3),(1234),(1234))}$.

There exist isotopisms $\Theta\in \mathfrak{I}_{n_1,\ldots,n_d,n}$ such that $\mathcal{PLH}_{\Theta}=\emptyset$. This is the case, for example, of the isotopism $((12),{\rm id},{\rm 
id})\in\mathfrak{I}_{2,2,2}$. Necessary conditions for isotopisms of (partial) Latin squares to be an autotopism are exposed in \cite{Falcon2012, Falcon2013, McKay2005, Sade1968,  Stones2012} and a classification of autotopisms of Latin squares of order $n\leq 17$ according to their cycle structures appear in \cite{Falcon2012, Stones2012}. Recall in this regard that the {\em cycle structure} of a permutation $\pi\in S_m$ is the expression  
\[z_{\pi}=m^{\lambda_m^{\pi}}\ldots 1^{\lambda_1^{\pi}},\] 
where $\lambda_l^{\pi}$ is the number of cycles of length $l$ in the decomposition of $\pi$ as a product of disjoint cycles. 
In practice, we only write those $l^{\lambda_l^{\pi}}$ for which $\lambda_l^{\pi}>0$. Besides, any term of the form $l^1$ is replaced by $l$. Thus, for instance, the cycle structure of the permutation $(123)(4)(567)\in S_7$ is $3^21$. Two permutations $\pi_1$ and $\pi_2$ in $S_m$ have the same cycle structure if and only if they are conjugate, that is, there exists a third permutation $\pi_3\in S_m$ such that $\pi_2=\pi_3\pi_1\pi_3^{-1}$. As a natural generalization, the {\em cycle structure} of an isotopism $\Theta\in \mathfrak{I}_{n_1,\ldots,n_d,n}$ is defined as the $(d+1)$-tuple $z_{\Theta}=(z_{\pi_1},\ldots,z_{\pi_d},z_{\pi})$. Thus, for instance, the cycle structure of the isotopism $((12)(34)(56),(123)(456),(123)(4)(5)(6))\in \mathfrak{I}_{6,6,6}$ is $(2^3,3^2,31^3)$. Similarly to permutations, two isotopisms have the same cycle structure if and only if they are conjugate. Furthermore, analogously to the case of (partial) Latin rectangles \cite{Falcon2013, Falcon2015, Stones2012}, the next result holds.

\begin{lem}\label{lem_CS} Let $\Theta$ be an isotopism in $\mathfrak{I}_{n_1,\ldots,n_d,n}$. The cardinality of the set $\mathcal{PLH}_{\Theta}$ only depends on the cycle structure of $\Theta$.
\end{lem}

\begin{proof} Let $\Theta_1$ be an isotopism in $\mathfrak{I}_{n_1,\ldots,n_d,n}$ with the same cycle structure like that of $\Theta$. Then, $\Theta$ and $\Theta_1$ are conjugate and 
there 
exists $\Theta_2\in \mathfrak{I}_{n_1,\ldots,n_d,n}$ such that $\Theta_1=\Theta_2\Theta\Theta_2^{-1}$. It is straightforwardly verified that $\Theta$ is an autotopism of a given $P\in\mathcal{PLH}_{n_1,\ldots,n_d,n}$ if and only if $P^{\Theta_2}\in\mathcal{PLH}_{\Theta_1}$. Both sets $\mathcal{PLH}_{\Theta}$ and $\mathcal{PLH}_{\Theta_1}$ have, therefore, the same cardinality, because $P^{\Theta_2}\neq Q^{\Theta_2}$, for any two distinct partial Latin hyper-rectangles $P,Q\in\mathcal{PLH}_{\Theta}$.
\end{proof}

The {\em cell orbit} of a tuple $(i_1,\ldots,i_d)\in [n_1]\times\ldots\times [n_d]$ under the action of an isotopism $\Theta=(\pi_1,\ldots,\pi_d,\pi)\in \mathfrak{I}_{n_1,\ldots,n_d,n}$ is defined as the subset
\begin{equation}\label{eq_orbit}
\mathfrak{o}_{\Theta}((i_1,\ldots,i_d)):= \{(\pi^m_1(i_1),\ldots,\pi^m_d(i_d))\mid\, m\in\IN\}.
\end{equation}
This definition generalizes the notion of cell orbit that was introduced by Stones et al. \cite{Stones2012} for Latin squares. From here on, fixed a permutation $\pi\in S_m$ and a symbol $s\in [m]$, we denote by $l_{\pi,s}$ the length of the cycle $C$ in the unique decomposition of $\pi$ into disjoint cycles such that \mbox{$\pi(s)=C(s)$}. Thus, for instance, $l_{((12)(3),1)}=2$.

\begin{lem}\label{lem_card_orbit} The next results hold.
\begin{enumerate}[a)]
\item The cell orbit described in (\ref{eq_orbit}) coincides with the set
\[\mathfrak{o}_{\Theta}((i_1,\ldots,i_d)):= \{(\pi^m_1(i_1),\ldots,\pi^m_d(i_d))\mid\, 0<m\leq \mathrm{lcm}(l_{\pi_1,i_1},\ldots,l_{\pi_d,i_d})\},\]
whose elements are pairwise distinct.
\item Every isotopism $\Theta\in\mathfrak{I}_{n_1,\ldots,n_d,n}$ determines a partition of the set $[n_1]\times\ldots\times [n_d]$.
\end{enumerate}
\end{lem}

\begin{proof} Let $j\leq d$ be a positive integer. The first claim follows straightforwardly from the fact that $\pi_j^{l_{\pi_j,i_j}}(i_j)=i_j$ and $\pi_j^k(i_j)\neq\pi_j^l(i_j)$, for all pair of distinct positive integers $k,l\leq l_{\pi_j,i_j}$. The second claim holds because the cell orbit of a given tuple in $[n_1]\times\ldots\times [n_d]$ under the action of the isotopism $\Theta$ coincides with the cell orbit of each one of its elements.
\end{proof}

The partition of the set $[n_1]\times\ldots\times [n_d]$ into cell orbits under the action of an isotopism also determines a partition of the cells of any partial Latin hyper-rectangle in $\mathcal{PLH}_{n_1,\ldots,n_d,n}$. Thus, for instance, the partition of the set $[3]\times [4]\times [4]$ under the action of the isotopism $((12)(3),(1234),(1234))\in \mathfrak{I}_{3,4,4}$ is the set formed by the three cell orbits $\mathfrak{o}_1=\{(1,1),(2,2),(1,3),(2,4)\}$, $\mathfrak{o}_2=\{(2,1),(1,2),(2,3),(1,4)\}$ and $\mathfrak{o}_3=\{(3,1),(3,$ $2),(3,3),(3,4)\}$. Figure \ref{fig_orbits} illustrates the partition into cell orbits of the cells that this isotopism gives rise to any partial Latin rectangle in $\mathcal{PLH}_{3,4,4}$. In the figure, the cells related to each orbit have respectively been filled with by the symbols $\blacktriangle$, $\blacktriangledown$ and $\blacklozenge$.

\begin{figure}[htbp]
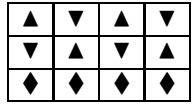

$$\begin{array}{|c|c|c|c|}
\hline
\blacktriangle&\blacktriangledown&\blacktriangle & \blacktriangledown\\ \hline
\blacktriangledown &\blacktriangle&\blacktriangledown & \blacktriangle\\ \hline
\blacklozenge&\blacklozenge&\blacklozenge & \blacklozenge\\
\hline
\end{array}$$
\caption{Orbit cells of any partial Latin rectangle in $\mathcal{PLH}_{3,4,4}$ by means of the isotopism $((12)(3),(1234),(1234))\in \mathfrak{I}_{3,4,4}$.} \label{fig_orbits}
\end{figure}

Let $\mathfrak{o}(\Theta)$ denote the partition of $[n_1]\times\ldots\times [n_d]$ by means of an isotopism \mbox{$\Theta\in\mathfrak{I}_{n_1,\ldots,n_d,n}$}. We say that a cell orbit 
in 
$\mathfrak{o}(\Theta)$ is {\em symbol-free} in a partial Latin hyper-rectangle $P\in\mathcal{PLH}_{n_1,\ldots,n_d,n}$ if all its elements correspond to empty cells in $P$. Otherwise, 
the cell orbit is said to be 
{\em marked}. It is called {\em complete} if all its elements correspond to non-empty cells in $P$. Thus, for instance, the previously mentioned cell orbits $\mathfrak{o}_1$ and 
$\mathfrak{o}_2$ are, respectively, complete and symbol-free in the partial Latin rectangle $P_3$ in Figure \ref{Fig1}. In the same figure, the cell orbit $\mathfrak{o}_1$ is marked in $P_2$, but it is not complete because the cell $(1,1)$ is empty. The next result follows straightforwardly from Lemma \ref{lem_card_orbit} and the notion of autotopism of a partial Latin hyper-rectangle.

\begin{prop}\label{prop_PLH} Let $\Theta=(\pi_1,\ldots,\pi_d,\pi)\in \mathfrak{I}_{n_1,\ldots,n_d,n}$ and $P\in\mathcal{PLH}_{n_1,\ldots,n_d,n}$. Then, $P\in\mathcal{PLH}_{\Theta}$ if and only if the next two conditions hold.
\begin{enumerate}[a)]
\item Every cell orbit in $P$ under the action of $\Theta$ is complete or symbol-free.
\item If $(i_1,\ldots,i_d)\in [n_1]\times\ldots\times [n_d]$ is a non-empty cell in $P$, then its cell orbit is formed by the non-empty cells $(\pi^m_1(i_1),\ldots,\pi^m_d(i_d))$, for 
every positive integer $m\leq \mathrm{lcm}(l_{\pi_1,i_1},\ldots,l_{\pi_d,i_d},l_{\pi,P(i_1,\ldots,i_d)})\}$, where
\begin{equation}\label{basiceqauto}
P(\pi_1^m(i_1),\ldots,\pi_d^m(i_d))=\pi^m(P(i_1,\ldots,i_d)).
\end{equation}
\end{enumerate}
\end{prop}

\vspace{0.5cm}

During the development of the colouring game based on an isotopism $\Theta=(\pi_1,\ldots,\pi_d,\pi)\in\mathfrak{I}_{n_1,\ldots,n_d,n}$ that is described in Section 3, Alice and Bob can deal with a partial Latin hyper-rectangle $P\in \mathcal{PLH}_{n_1,\ldots,n_d,n}$ with at least one cell orbit under the action of $\Theta$ that is neither complete nor symbol-free, but such that $E(P)\subset E(Q)$ for some $Q\in \mathcal{PLH}_{\Theta}$. Due to this fact, we introduce here the concepts of compatibility and feasibility.

\subsubsection{Compatibility}

We say that a tuple of positive integers $(i_1,\ldots,i_k)\in\mathbb{N}^k$ of {\em weight} $\sum_{j\leq k}i_j$ is {\em lcm-compatible} if the least common multiple of any $k-1$ 
elements in the ordered set $\{i_1,\ldots,i_k\}$ coincides with the least common multiple of all of them. Hereafter, we denote by $\mathcal{C}_k$ the set of lcm-compatible $k$-tuples. 
Thus, for instance, the tuple $(1,2,4,4)$ belongs to $\mathcal{C}_4$ because $\mathrm{lcm}(1,2,4)=\mathrm{lcm}(1,4)=\mathrm{lcm}(2,4)=4$. Similarly to the case of (partial) Latin 
rectangles (cf.~\cite{Falcon2013, Falcon2015, Stones2012}), the next result characterizes the autotopisms of the set $\mathcal{PLH}_{n_1,\ldots,n_d,n}$.

\pagebreak[3]

\begin{prop}\label{prop_PLH_0} Let $\Theta=(\pi_1,\ldots,\pi_d,\pi)$ be an isotopism in $\mathfrak{I}_{n_1,\ldots,n_d,n}$. A tuple $(i_1,\ldots,i_d,i)\in [n_1]\times\ldots\times 
[n_d]\times [n]$ can be the entry of a partial Latin hyper-rectangle in $\mathcal{PLH}_{\Theta}$ if and only if $(l_{\pi_1,i_1},\ldots,l_{\pi_d,i_d},l_{\pi,i}) \in\mathcal{C}_{d+1}$. 
If this is the case, then 
\[\{(\pi_1^m(i_1),\ldots,\pi_d^m(i_d),\pi^m(i))\mid\, 0<m\leq \mathrm{lcm}(l_{\pi_1,i_1},\ldots,l_{\pi_d,i_d},l_{\pi,i})\}\] 
constitutes the set of entries of a partial Latin 
hyper-rectangle in $\mathcal{PLH}_{\Theta}$.
\end{prop}

\begin{proof} Let $(i_1,\ldots,i_d,i)$ be an entry of some \mbox{$P\in\mathcal{PLH}_{\Theta}$}. From Lemma \ref{lem_card_orbit}, $\{(\pi_1^m(i_1),\ldots,$ 
$\pi_d^m(i_d),\pi^m(i)\mid\, m\in\mathbb{N}\}$ is a subset of $E(P)$ that is related to exactly $\mathrm{lcm}(l_{\pi_1,i_1},\ldots,l_{\pi_d,i_d})$ distinct cells of $P$. From the Latin array condition, this coincides with the least common multiple of the tuple that results after replacing any of the components $l_{\pi_j,i_j}$, with $j\leq d$, by $l_{\pi,i}$ and hence, with $\mathrm{lcm}(l_{\pi_1,i_1},\ldots,l_{\pi_d,i_d},l_{\pi,i})$. This is equivalent to say that the tuple $(l_{\pi_1,i_1},\ldots,l_{\pi_d,i_d},l_{\pi,i})$ belongs to $\mathcal{C}_{d+1}$. Otherwise, either there would appear twice the symbol $i$ in a same line of the array $P$ or there would exist a cell with at least two distinct assigned symbols, which is not possible.
\end{proof}

Let $\Theta=(\pi_1,\ldots,\pi_d,\pi)\in \mathfrak{I}_{n_1,\ldots,n_d,n}$. We say that a partial Latin hyper-rectangle $P\in \mathcal{PLH}_{n_1,\ldots,n_d,n}$ is {\em $\Theta$-compatible} if the next two conditions hold.
\begin{enumerate}[C.1)]
\item The tuple $(l_{\pi_1,i_1},\ldots,l_{\pi_d,i_d},l_{\pi,P(i_1,\ldots,i_d)})$ belongs to $\mathcal{C}_{d+1}$, for every non-empty cell $(i_1,$ $\ldots,i_d)$ in $P$.
\item The cell $(\pi_1^m(i_1),\ldots,\pi_d^m(i_d))$ in $P$ is either empty or Condition (\ref{basiceqauto}) holds, for every non-empty cell $(i_1,\ldots,i_d)$ in $P$ and for every 
positive integer $m\in \mathbb{N}$.
\end{enumerate}

The reasoning exposed in the first part of the proof of Proposition \ref{prop_PLH_0} enables us to ensure that every partial Latin hyper-rectangle in $\mathcal{PLH}_{\Theta}$ is $\Theta$-compatible. Nevertheless, the reciprocal does not hold in general. Observe, for instance, that the partial Latin rectangles $P_2$ and $P_3$ in Figure \ref{Fig1} are $((12)(3),(1234),(1234))$-compatible, whereas $P_1$ is not. Besides, $P_3$ belongs to $\mathcal{PLH}_{((12)(3),(1234),(1234))}$, whereas $P_2$ does not. The next result characterizes the set $\mathcal{PLH}_{\Theta}$ by means of $\Theta$-compatibility. This follows straightforwardly from Propositions \ref{prop_PLH} and \ref{prop_PLH_0}.

\begin{thm}\label{thm_PLH} Let $\Theta\in \mathfrak{I}_{n_1,\ldots,n_d,n}$. Then, $\mathcal{PLH}_{\Theta}\neq \emptyset$ if and only if there exists 
$P\in\mathcal{PLH}_{n_1,\ldots,n_d,n}$ that is $\Theta$-compatible and for which for all its entries (\ref{basiceqauto}) holds.
\end{thm}

\subsubsection{Feasibility}

We say that an isotopism $\Theta=(\pi_1,\ldots,\pi_d,\pi)\in \mathfrak{I}_{n_1,\ldots,n_d,n}$ is {\em feasible} for a colouring game based on the Hamming graph 
$\mathcal{H}_{n_1,\ldots,n_d}$, or shortly, {\em feasible}, if $(l_1,\ldots,l_d,\ell)\in\mathcal{C}_{d+1}$ for every tuple $(l_1,\ldots,l_d,\ell)\in [n_1]\times\ldots\times
[n_d]\times [n]$ such that $\lambda_{\ell}^{\pi}\cdot \prod_{j=1}^d\lambda_{l_j}^{\pi_j}>0$. Thus, for instance, the isotopism 
\[((12)(34)(56),(123)(456),(123456),(123)(45)(6))\in\mathfrak{I}_{6,6,6,6}\] 
is feasible because its cycle structure is $(2^3,3^2,$ $6,321)$ and the tuples $(2,3,6,3)$, $(2,3,6,2)$, 
and $(2,3,6,1)$ belong to $\mathcal{C}_4$. The next result follows straightforwardly from the just exposed notion of feasibility.

\begin{lem}\label{lem_feasible} Let $\Theta=(\pi_1,\ldots,\pi_d,\pi)\in\mathfrak{I}_{n_1,\ldots,n_d,n}$. The next results hold.
\begin{enumerate}[a)]
\item If $\Theta$ is feasible and $\prod_{j=1}^d\lambda_1^{\pi_j}>0$, then $\pi$ is the trivial permutation in $S_n$.
 \item If the cycle structure of $\Theta$ is $z_{\Theta}=(l_1^{n_1/l_1},\ldots, l_d^{n_d/l_d}, {\ell}^{n/\ell})$ for some $(d+1)$-tuple $(l_1,\ldots,l_d,\ell)\in 
[n_1]\times\ldots\times [n_d]\times [n]$, then $\Theta$ is feasible if and only if $(l_1,\ldots,l_d,\ell)\in\mathcal{C}_{d+1}$.
\end{enumerate}
\end{lem}

\vspace{0.1cm}

We define an \emph{extension} of {\em size} $n'\geq n$ of an isotopism $\Theta=(\pi_1,\ldots,\pi_d,\pi)\in \mathfrak{I}_{n_1,\ldots,n_d,n}$ as another isotopism 
$\Theta'=(\pi_1,\ldots,\pi_d,\pi')\in \mathfrak{I}_{n_1,\ldots,n_d,n'}$, such that $\pi'(s)=\pi(s)$, for all $s\leq n$. An extension is called \emph{natural} if $\pi'(s)=s$, for all 
$n<s\leq n'$. The 
isotopism $\Theta$ is, therefore, a trivial natural extension of itself. We say that an isotopism is \emph{extendable} if all its natural extensions are feasible. Particularly, extendability involves feasibility. The next results deep further into both concepts.

\begin{prop}\label{prop_Board} If a natural extension of an isotopism is feasible, then the latter is feasible. This is extendable whenever the former is not trivial or the isotopism is principal. Particularly, every feasible principal isotopism is extendable.
\end{prop}

\begin{proof} Let $\Theta=(\pi_1,\ldots,\pi_d,\pi)$ and $\Theta'=(\pi_1,\ldots,\pi_d,\pi')$ be respective isotopisms in $\mathfrak{I}_{n_1,\ldots,n_d,n}$ and 
$\mathfrak{I}_{n_1,\ldots,n_d,n'}$ such that $\Theta'$ is feasible and a natural extension of~$\Theta$. Since $\pi'(s)=\pi(s)$, for all $s\leq n$, the feasibility of $\Theta'$ 
involves that of~$\Theta$. Under these conditions, it is straightforwardly verified that, if $\pi=\mathrm{Id}$, then $\Theta$ is extendable. In the general case, suppose that $n'>n$. 
Then, $\lambda_1^{\pi'}>0$ and Lemma~\ref{lem_feasible} involves that $(l_1,\ldots,l_d,1)\in\mathcal{C}_{d+1}$, for all $(l_1,\ldots,l_d)\in [n_1]\times\ldots \times [n_d]$ such that 
$\prod_{j=1}^d\lambda_{l_j}^{\pi_j}>0$. This condition is shared by any natural extension of $\Theta$, which becomes, therefore, extendable. The last assertion follows immediately 
from the fact that every isotopism is a trivial natural extension of itself.
\end{proof}

\vspace{0.1cm}

\begin{thm}\label{thm_Board} An isotopism $(\pi_1,\ldots,\pi_d,\pi)\in\mathfrak{I}_{n_1,\ldots,n_d,n}$ is extendable if and only if it is feasible and one of the next two assertions hold.
\begin{enumerate}[a)]
\item $d=2$ and all the cycles in the unique decompositions of $\pi_1$ and $\pi_2$ into a product of disjoint cycles have the same length.
\item $d>2$ and the isotopism $(\pi_1,\ldots,\pi_d)\in\mathfrak{I}_{n_1,\ldots,n_d}$ is feasible.
\end{enumerate}
\end{thm}

\begin{proof} The result follows straightforwarly from Proposition \ref{prop_Board} and the fact that a tuple $(l_1,\ldots,l_d,1)\in\mathcal{C}_{d+1}$ if and only if 
$(l_1,\ldots,l_d)\in\mathcal{C}_d$.
\end{proof}

\vspace{0.1cm}

Thus, for instance, the isotopism $((12),(12)(34),(12)(3))\in\mathfrak{I}_{2,4,3}$ is extendable, because it is feasible and all the cycles of the permutations of rows and columns 
have length $2$. The isotopism 
\[((123),(12)(34),(123456),(123)(45))\in\mathfrak{I}_{3,4,6,5}\] 
is also extendable, because itself and 
$((123),(12)(34),(123456))\in\mathfrak{I}_{3,4,6}$ are feasible. However, the latter is  not extendable, because the pair $(3,2)\not\in\mathcal{C}_2$.

\section{The game}\label{secgame}

 Let $\Theta$ be an extendable isotopism in $\mathfrak{I}_{n_1,\ldots, n_d,n}$ and let $a$ and $b$ be two positive integers. We introduce here the so-called \emph{$\Theta$-stabilized $(a,b)$-colouring game} that is 
 played on the Hamming graph $\mathcal{H}_{n_1,\ldots, n_d}$ with regard to $\Theta$, by two players, Alice and Bob, and a given number $n'\ge n$ of colours. As in the conventional colouring game, if all vertices of the 
 Hamming graph $\mathcal{H}_{n_1,\ldots, n_d}$ are coloured at the end of the game, then Alice wins, otherwise Bob wins. At the beginning of the game we consider as board the trivial partial Latin hyper-rectangle 
 \mbox{$P\in\mathcal{PLH}_{n_1,\ldots, n_d}$}, with all its cells being empty. Alternately the players choose an empty cell $(i_1,\ldots,i_d)\in [n_1]\times\ldots\times [n_d]$ in the board and a symbol $s\in [n']$ and 
colour 
 the former by the latter by setting $P(i_1,\ldots,i_d):=s$, where Alice makes $a$ turns (choices and colourings), whereas Bob makes $b$ turns. The colouring $s$ of the cell $(i_1,\ldots,i_d)$ has to obey the next three 
 rules (see Figure \ref{rulefigure} for illustrative examples).
\begin{quote}
\begin{enumerate}
\item[\bf(Rule 1)] The Latin array condition must hold.
\item[\bf(Rule 2)] The $\Theta$-compatibility condition must hold.
\item[\bf(Rule 3)] For each positive integer $m<\mathrm{lcm}(l_{\pi_1,i_1},\ldots,l_{\pi_d,i_d})$ and each coloured collinear cell $(i_1',\ldots,i_d')$ of $(\pi_1^m(i_1),\ldots,\pi_d^m(i_d))$, it must be
\begin{equation}\label{extendedneighbourGauto}
P(i_1',\ldots,i_d')\neq \pi^m(s).
\end{equation}
\end{enumerate}
\end{quote}

\renewcommand{\tabcolsep}{2pt}
\begin{figure}[htbp]
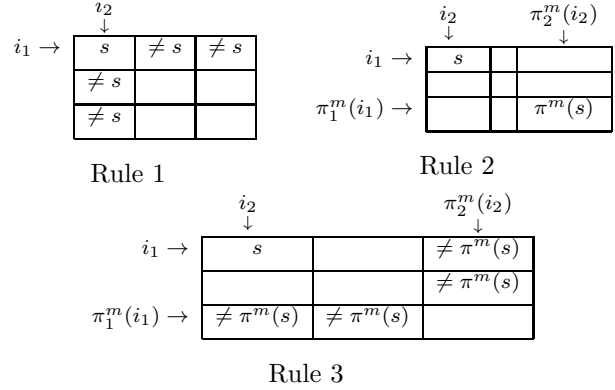

\footnotesize
\begin{center}
\begin{tabular}{ccc}
\begin{tabular}{c}
$\begin{array}{r|c|c|c|}
\multicolumn{1}{c}{}&\multicolumn{1}{c}{\underset{\downarrow}{i_2}}&
\multicolumn{1}{c}{}&
\multicolumn{1}{c}{}\\[1ex]
\cline{2-4}
i_1\rightarrow& s &\neq s&\neq s\\[1ex]
\cline{2-4}
&\neq s&&\\[1ex]
\cline{2-4}
&\neq s&&\\[1ex]
\cline{2-4}
\end{array}$\\ \\
\normalsize Rule 1
\end{tabular}
& \hspace{0.2cm} &
\begin{tabular}{c}
$\begin{array}{r|c|c|c|}
\multicolumn{1}{c}{}&\multicolumn{1}{c}{\underset{\downarrow}{i_2}\hspace{0.25cm}}&
\multicolumn{1}{c}{} &
\multicolumn{1}{c}{\underset{\downarrow}{\pi_2^m(i_2)}}\\[1ex]
\cline{2-4}
i_1\rightarrow&s&&\\ \cline{2-4}
& \,\, &&\\ \cline{2-4}
\pi_1^m(i_1)\rightarrow&&&\pi^m(s)\\[1ex]
\cline{2-4}
\end{array}$\\ \\
\normalsize Rule 2
\end{tabular}
\end{tabular}

\begin{tabular}{c}
$\begin{array}{r|c|c|c|}
\multicolumn{1}{c}{}&\multicolumn{1}{c}{\underset{\downarrow}{i_2}\hspace{0.25cm}}&
\multicolumn{1}{c}{}&
\multicolumn{1}{c}{\underset{\downarrow}{\pi_2^m(i_2)}}\\[1ex]
\cline{2-4}
i_1\rightarrow&s&&\neq \pi^m(s)\\[1ex]
\cline{2-4}
&&&\neq 
{\pi^m(s)}\\[1ex]
\cline{2-4}
\pi_1^m(i_1)\rightarrow&\neq \pi^m(s)&\neq \pi^m(s)&\\[1ex]
\cline{2-4}
\end{array}$\\ \\
\normalsize Rule 3
\end{tabular}

\end{center}
\caption{\label{rulefigure}Illustration of the three rules of the $\Theta$-stabilized $(a,b)$-colouring game.}
\end{figure}

 These rules enable us to ensure that any feasible turn of Alice and Bob consists of colouring an empty cell in a symbol-free cell orbit of the board by respecting Rules 1 and 3, or in a marked cell orbit by keeping in 
 mind Rule 2. The only possible colouring in this last case would then be forced by Condition~(\ref{basiceqauto}). This is the main idea on which is based the proof of the next result.

\begin{lem}\label{lem_AliceWins} Alice always wins the $\Theta$-stabilized $(a,b)$-colouring game when a player colours any cell of the last symbol-free cell orbit of the board.
\end{lem}

\begin{proof} Once an empty cell of the last symbol-free cell orbit of the board is coloured, the colouring of all the empty cells of the board is uniquely determined by means of Rule 2. This mandatory colouring determines indeed a Latin hyper-rectangle. Otherwise, there would exist two collinear cells that would have to be coloured with the same colour. Nevertheless, this situation involves that at least one of the previous moves would not have been allowed by Rules 1 or 3.
\end{proof}

\vspace{0.4cm}

The next result enables us to ensure that the extendability of the isotopism~$\Theta$ is required to get a well-defined colouring game. Specifically, if the game has not finished in the sense that there exists at least 
one symbol-free cell orbit, then any empty cell can be coloured and any colour related to a given extension of $\Theta$ that has not yet been used in the development of the game can always be employed in a feasible move.

\begin{prop}\label{extensionsareuseful} Let $\Theta$ be an extendable isotopism and let $P$ be a $\Theta$-compatible partial Latin hyper-rectangle that satisfies Rule 3 and has at 
least one 
symbol-free cell orbit under the action of $\Theta$. Then, any empty cell in $P$ can be coloured. Further, if a symbol $s$ related to an extension of $\Theta$ does not appear in any cell of $P$, then there exists at least one empty cell in $P$ that can be coloured with the colour $s$ by obeying Rules 1--3.
\end{prop}

\begin{proof} Let us consider an empty cell in $P$. If this is contained in a marked cell orbit under the action of $\Theta$, then Rule 2 forces its colour. Besides, Rule 3 guarantees that this forced colouring does not contradict Rule 1. Otherwise, if the cell orbit is symbol-free, then, from Proposition~\ref{extensionsareuseful}, colouring the cell with a new colour is feasible because the isotopism $\Theta$ is extendable.

Now, suppose $\Theta'=(\pi_1,\ldots,\pi_d,\pi)$ be an extension of $\Theta$ and let $s$ be a symbol that does not appear in any cell of $P$. Exactly one of the next situations holds.
\begin{enumerate}[a)]
\item $\pi(s)=s$. In such a case, it is enough to colour any cell of a symbol-free cell orbit of $P$ with the colour $s$.
\item $\pi(s)\neq s$ and there exists a marked cell orbit in $P$ containing a symbol $s'\neq s$ such that $\pi^m(s')=s$, for some $m\in\mathbb{N}$. From Rules 2--3, there exists an entry $(i_1,\ldots,i_n,s')\in E(P)$ such that the cell $(\pi^m_1(i_1),\ldots,\pi^m_d(i_d))$ is empty. It is then enough to colour the latter with the colour $s$.
\item $\pi(s)\neq s$ and there does not exist a marked cell orbit as in (b). It is then enough to colour any cell of a symbol-free cell orbit of $P$ with the colour $s$.
\end{enumerate}

Observe that, in any of the exposed cases, the colouring of the corresponding cell with the colour $s$ does not contradict Rules 1--3.
\end{proof}

\vspace{0.2cm}

Rule 3 is also required to have our game nice properties. To see it, let us call \emph{first-try-$\Theta$-stabilized $(a,b)$-colouring game} the game that results of eliminating this third rule. The next example shows the existence of configurations for which the corresponding chromatic number of this new game is not finite.

\begin{example}\label{firsttryexample} Let $\Theta=((123),(123)(456),{\rm Id})$, which is an extendable isotopism in $\mathfrak{I}_{3,6,6}$, and let $\Theta'$ be the natural extension 
of 
$\Theta$ of size $n'\ge 6$. In any feasible colouring of the Hamming graph $\mathcal{H}_{3,6}$, the cycle structure of $\Theta'$ involves the existence of six {\em circulant} cell orbits (see Figure \ref{Fig_5}, where the cells related to each orbit have respectively been filled with by the symbols $\blacktriangle$, $\blacktriangledown$, $\blacklozenge$, $\triangle$, $\triangledown$ and $\diamondsuit$).

\begin{figure}[htbp]
\begin{center}
$\begin{array}{|c|c|c|c|c|c|}
\hline
\blacktriangle&\blacktriangledown&\blacklozenge & \triangle & \triangledown & \diamondsuit\\
\hline
\blacklozenge&\blacktriangle&\blacktriangledown & \diamondsuit & \triangle & \triangledown\\
\hline
\blacktriangledown&\blacklozenge&\blacktriangle & \triangledown & \diamondsuit & \triangle\\
\hline
\end{array}$
\end{center}
\caption{Cell orbits of $\Theta=((123),(123)(456),{\rm Id})\in \mathfrak{I}_{3,6,6}$.}\label{Fig_5}
\end{figure}

Consider the first-try-$\Theta$-stabilized colouring game, with player Alice beginning, which is played on the Hamming graph $\mathcal{H}_{3,6}$ with regard to $\Theta'$. Alice 
colours 
w.l.o.g. the cell $(1,1)$ with the colour $c$. Then, Bob colours the cell $(2,4)$ with the same colour $c$. This is a feasible move according to the first-try definition. However, the cell $(2,2)$ cannot be coloured any more, since it should be coloured $c$ due to Rule 2, but it should be coloured with a colour distinct of $c$ due to Rule 1 (see Figure \ref{Fig_6}). Therefore, Bob would win for any number of colours.

\begin{figure}[htbp]
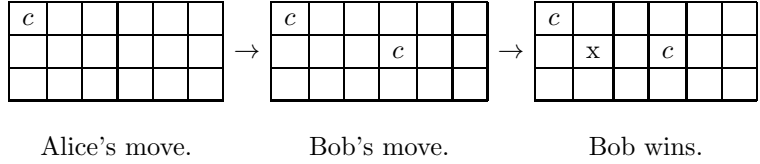

\begin{center}
\begin{tabular}{ccccc}
$\begin{array}{|c|c|c|c|c|c|}
\hline
c & \ &  \  &  \  & \  & \  \\ \hline
 & & & & &\\  \hline
 & & & & &\\  \hline
\end{array}$ & $\rightarrow$ & $\begin{array}{|c|c|c|c|c|c|}
\hline
c & \ &  \  &  \  & \  & \  \\ \hline
 & & & c & &\\  \hline
 & & & & &\\  \hline
\end{array}$& $\rightarrow$ & $\begin{array}{|c|c|c|c|c|c|}
\hline
c & \ &  \  &  \  & \  & \  \\ \hline
 & \mathrm{x} & & c & &\\  \hline
 & & & & &\\  \hline
\end{array}$\\ \ \\
Alice's move. & & Bob's move. & & Bob wins.\\
\end{tabular}
\end{center}
\caption{Winning strategy for Bob.}\label{Fig_6}
\end{figure}

We remark that Bob's destroying move is feasible for the first-try-$\Theta$-stabilized colouring game, but it is not feasible in the $\Theta$-stabilized colouring game since it 
contradicts 
Rule 3.
\end{example}

\hspace{0.1cm}

The smallest size $n'$ of the natural extension of the isotopism $\Theta$ for which Alice has a winning strategy in our original game is called {\em $\Theta$-stabilized $(a,b)$-game chromatic number} of $\mathcal{H}_{n_1,\ldots,n_d}$ and is denoted as $\gcnt{a}{b}{\mathcal{H}_{n_1,\ldots,n_d}}{}$, or ${}^{(a,b)}\chi_g^{\Theta}$ for short. In case of being $a=b=1$, it is called {\em $\Theta$-stabilized game chromatic number}. The parameter ${}^{(1,1)}\chi_g^{\Theta}$ is also denoted as $\chi_g^{\Theta}$. If, besides, $\Theta$ is the trivial isotopism, then this corresponds to the usual game chromatic number $\chi_g(\mathcal{H}_{n_1,\ldots,n_d})$.

\begin{prop}\label{prop_CS} Let $\Theta_1$ and $\Theta_2$ be two extendable isotopisms with the same cycle structure and let $a$ and $b$ be two positive integers. Then, 
\[{}^{(a,b)}\chi_g^{\Theta_1}={}^{(a,b)}\chi_g^{\Theta_2}.\]
\end{prop}

\begin{proof} The result is based on Lemma \ref{lem_CS}. Particularly, since $\Theta_1$ and $\Theta_2$ have the same cycle structure, there exists an isotopism $\Theta$ such that $\Theta_2=\Theta\Theta_1\Theta^{-1}$. The result follows straightforwardly from the fact that the winning strategy of Alice for the $\Theta_2$-stabilized $(a,b)$-colouring game is exactly the same of that for the $\Theta_1$-stabilized $(a,b)$-colouring game. Specifically, every partial Latin hyper-rectangle $P\in\mathcal{PLH}_{\Theta_1}$ that corresponds to a position of the winning strategy of Alice for the latter is uniquely related to the partial Latin hyper-rectangle $P^{\Theta}\in\mathcal{PLH}_{\Theta_2}$ that corresponds to the analogous position of the winning strategy of Alice for the former.
\end{proof}

\vspace{0.1cm}

\begin{lem} \label{lem_Bounds} Let $a$ and $b$ be two positive integers and let $\Theta$ be an extendable isotopism in $\mathfrak{I}_{n_1,\ldots,n_d,n}$. Then,
\begin{enumerate}[a)]
\item ${}^{(a,b)}\chi_g^{\Theta}={}^{(|\mathfrak{o}(\Theta)|,b)}\chi_g^{\Theta}$, for all $a\geq |\mathfrak{o}(\Theta)|$.
\item ${}^{(a,b)}\chi_g^{\Theta}={}^{(a,|\mathfrak{o}(\Theta)|)}\chi_g^{\Theta}$, for all $b\geq |\mathfrak{o}(\Theta)|$.
\item If $n=\max_{j=1,\ldots,d} n_j$, then
\[n\le {}^{(a,b)}\chi_g^{\Theta}\le |\mathfrak{o}(\Theta)|+n-1.\]
\end{enumerate}
\end{lem}

\begin{proof} Under the assumptions of (a) and (b), since the number of moves of both players is at least the number of cell orbits under the action of $\Theta$, they can impose the colouring of all the cells by colouring just one cell of each orbit. Hence, the winning strategy of Alice, if one exists at all, is the same for any such a number of moves in both cases.

The lower bound in (c) holds straightforwardly. Now, since extendability involves feasibility, the game can start because any cell of the empty partial Latin hyper-rectangle in 
$\mathcal{PLH}_{n_1,\ldots,n_d,n}$ can be coloured with any of the symbols of the set $[n]$. Besides, Rule 2 enables us to ensure that at least the first cell orbit that is chosen in the first move can also be coloured with the symbols in~$[n]$. Since the colouring of any other orbit cell is uniquely determined by that of any of its cells, the upper bound results from the fact that we can ensure the complete colouring of $P$ by considering an extension of the isotopism $\Theta$ with at most $|\mathfrak{o}(\Theta)|-1$ new distinct colours.
\end{proof}

\vspace{0.1cm}

 The lower bound in item (c) of Lemma \ref{lem_Bounds} is tight, for instance, for the isotopism $\Theta_1=((12),(12),\mathrm{Id})\in\mathfrak{I}_{2,2,2}$, whereas the upper bound is tight for the isotopism 
 $\Theta_2=((12),(12),(12))\in\mathfrak{I}_{2,2,2}$. In both games, the first player can start w.l.o.g. by colouring the cell $(1,1)$ of the empty partial Latin square of order $2$ with the symbol $1$. In the first case, 
 the cell $(2,2)$ must also be coloured with the symbol~$1$, whereas the cells $(1,2)$ and $(2,1)$ must be coloured with the symbol~$2$. In the second case, the cell $(2,2)$ must be coloured with the symbol~$2$ and the 
 Latin array condition involves the cells $(1,2)$ and $(2,1)$ to be coloured with a third colour $3$ related to the natural extension 
\[\Theta'_2=((12),(12),(12)(3))\in\mathfrak{I}_{2,2,3}.\]

\vspace{0.3cm}

 We can consider two variants, $g_A$ and $g_B$, of the proposed game $g$ depending, respectively, on whether Alice or Bob does the first move. To make clear which variant we refer, we denote the corresponding chromatic 
 numbers with the subindices $g_A$ or $g_B$ instead of $g$. The specific case of the variant $g_B$ for which $\Theta$ is the trivial isotopism, $d=2$, and $n_1=n_2=n$ corresponds to 
what Schlund~\cite{Schlund2011} called 
 {\em $(n,a,b)$-game}. He proved in particular the next result.

\begin{prop}[Schlund \cite{Schlund2011}]\label{Schlund} Let $n$ be a positive integer. Then,
\begin{enumerate}[a)]
\item Bob wins the $(n,1,1)$-game, for all $n\geq 3$.
\item Alice wins the $(n,n-1,1)$-game.
\item If Alice wins the $(n,a,1)$-game, then she also wins the $(n,2a+1,1)$-game.
\item Alice wins the $(n,2^kn-1,1)$-game for all positive integers $k$.
\item $n+1\leq \chi_{g_B}(\mathcal{H}_{n,n})$, for all $n\geq 3$.
\end{enumerate}
\end{prop}

\vspace{0.2cm}

The next result enables us to ensure that the lower bound exposed in the last assertion in Proposition \ref{Schlund} is tight for $n=3$.

\begin{prop} \label{prop_Hnn} $\chib(\mathcal{H}_{3,3})=4$.
\end{prop}

\begin{proof} Suppose that Bob starts a $(3,1,1)$-game with  four distinct colours. He wins the game if and only if there exists a configuration during the game with an empty cell having all its four neighbours with 
distinct colours. Due to it, Bob must always avoid the colouring  of three cells with the same colour. Let us expose here a possible winning strategy for Alice. W.l.o.g. we can suppose that Bob starts the game by 
colouring the cell $(1,1)$ of the empty partial Latin square of  order $3$ with colour 1. Then, Alice must colour a cell distinct of $(2,2)$ and $(3,3)$. Otherwise, a case study enables us to ensure that Bob has a 
winning strategy. We can suppose, therefore, that Alice colours  the cell $(1,2)$ with colour~2. Now, we can suppose that Bob uses a colour $c\in\{3,4\}$. Otherwise, Alice could colour a third cell with the same colour 
$1$ or $2$ and would win the game. From here on, we suppose that  $c=3$. Whatever Bob's second move is, Alice can colour a cell in the third column with colour~1. The only configurations that Alice must avoid under such 
conditions are, up to permutation of the second and third rows,
$$\begin{array}{|c|c|c|} \hline
1 & 2 & \ \\ \hline
3 & \ & 1 \\ \hline
\ & \ & \ \\ \hline
\end{array} \hspace{1cm}\text{ and } \hspace{1cm}
\begin{array}{|c|c|c|} \hline
1 & 2 & \ \\ \hline
\ & 3 & \ \\ \hline
\ & \ & 1 \\ \hline
\end{array}$$
In both cases, Bob would win the game by colouring, respectively, the cell $(3,2)$ or $(3,1)$ with colour~4. Once these two configurations are avoided, the second move of Alice forces the third one of Bob, who must colour with a colour $c'\not\in\{1,2\}$ the unique cell in the second column that would make possible the third use of the colour~1. Up to isotopism, the possible configurations of the game at this moment are
$$\begin{array}{|c|c|c|} \hline
1 & 2 & \ \\ \hline
3 & 4 & \ \\ \hline
\ & \ & 1 \\ \hline
\end{array} \hspace{0.5cm}
\begin{array}{|c|c|c|} \hline
1 & 2 & \ \\ \hline
\ & 3 & 1 \\ \hline
\ & 4 & \ \\ \hline
\end{array}\hspace{0.5cm}
\begin{array}{|c|c|c|} \hline
1 & 2 & 3 \\ \hline
\ & \ & 1 \\ \hline
\ & 3 & \ \\ \hline
\end{array}\hspace{0.5cm}
\begin{array}{|c|c|c|} \hline
1 & 2 & 3 \\ \hline
\ & \ & 1 \\ \hline
\ & 4 & \ \\ \hline
\end{array}
\hspace{0.25cm} \text{ or } \hspace{0.25cm}
\begin{array}{|c|c|c|} \hline
1 & 2 & \ \\ \hline
\ & 4 & 3 \\ \hline
\ & \ & 1 \\ \hline
\end{array}$$
 A simple case study involves Alice to win the game based on the first configuration and to guarantee her victory in the remaining ones by colouring, respectively, the cells $(3,1)$, $(2,1)$, $(3,1)$ and $(3,1)$  with 
the colours $3$, $3$, $3$ and~$4$.
\end{proof}

Schlund~\cite{Schlund2011} also indicated (loc. cit. page 57) that $n+1\leq \chi_{g_A}(\mathcal{H}_{n,n})$, for all $n\geq 3$. Nevertheless, the next result involves this lower bound 
to be wrong 
for $n=3$.

\begin{prop} \label{prop_Hnn_2} $\chia(\mathcal{H}_{3,3})=3$.
\end{prop}

\begin{proof} A winning strategy for Alice with 3 colours is the following. W.l.o.g. Alice colours the cell $(1,1)$ of the empty partial Latin square of order $3$ with colour 1. Due to symmetry there are only three relevant cases to consider.
\begin{itemize}
\item {\bf Case 1}. Bob colours the cell $(1,2)$ with colour 2. Then, Alice responds by colouring the cell $(1,3)$ with colour 3. By symmetry, w.l.o.g. Bob colours the cell $(2,1)$ with colour 2. Then, Alice may fix the colouring and wins by colouring the cell $(3,2)$ with colour~1.
\item {\bf Case 2}. Bob colours the cell $(2,2)$ with colour 1. Then, Alice colours the cell $(3,3)$ with colour~1. By symmetry, w.l.o.g. Bob colours $(1,2)$ with colour~2. Then, Alice fixes the colouring by colouring $(2,1)$ with colour~3.
\item {\bf Case 3}. Bob colours the cell $(2,2)$ with colour 2. Then, Alice fixes the colouring by colouring the cell $(3,3)$ with colour~3.
\end{itemize} \vspace{-0.8cm}
\end{proof}

\vspace{0.2cm}

Propositions \ref{prop_Hnn} and \ref{prop_Hnn_2} refer to the colouring game of the small Hamming graph~$\mathcal{H}_{3,3}$. Let us finish this section with some other results related 
to 
the Hamming graph $\mathcal{H}_{n_1,n_2}$, with $n_1\leq 2\leq n_2$. The next result is useful to this end. This is based on a previous idea of Andres~\cite{Andres2003}, who proved 
that 
\[\chib(K_2\Box TG)\le\Delta(K_2\Box TG),\] 
for any toroidal grid graph $TG$ (see loc. cit. Lemma 17).

\begin{lem}\label{helpforKtwo} Let $G=(V,E)$ be a graph with $|E|\neq\emptyset$. Then,
\[\chib(K_2\Box G)\le \Delta(K_2\Box G))\leq \Delta(G)+1.\]
\end{lem}

\begin{proof}
Let $\Delta:=\Delta(K_2\Box G)$ and let $M:=\IZ/\Delta\IZ$ be a set of $\Delta$ colours that we consider as additive group. In particular, $\Delta\geq 2$, because $|E|\neq\emptyset$. Let $\{-1,+1\}$ be the vertex set of $K_2$ and let $\mathrm{colour}((a,v))$ denote the colour of a coloured vertex $(a,v)\in K_2\Box G$. Alice's winning strategy with $|M|$ colours is as follows. Whenever Bob colours a vertex $(a,v)\in\{-1,+1\}\times V$ with colour $m\in M$, she colours the vertex $(-a,v)$ with colour $m+a\mod\Delta$. This is different from $m$ because $1\neq0$ in~$M$. Hence, after Alice's moves, for any $w\in V$, either $(+1,w)$ and $(-1,w)$ are both coloured or they are both uncoloured. This means that, whenever Bob colours a vertex $(a,v)$, this vertex has an uncoloured neighbour, namely $(-a,v)$. There are, therefore, at most $\Delta(G)$ coloured neighbours and hence, there is at least one feasible colour for Bob's move. If $(a,v)$ is coloured with colour $m$, then none of the vertices $(a,w)$ with $w\neq v$ and $vw\in E$ is coloured with $m$. By Alice's
strategy, after her moves, for any coloured vertex $(+1,w)$, we have the invariant $\mathrm{colour}((+1,w))+1=\mathrm{colour}((-1,w))\mod\Delta$. Therefore, none of the vertices $(-a,w)$ with $w\neq v$ and $vw\in E$ is coloured with $m+a\mod\Delta$. Thus Alice's move is always feasible.
\end{proof}



\pagebreak[4]

\begin{prop} The next results hold.
\begin{enumerate}[a)]
\item $\chia(\mathcal{H}_{1,n})=\chib(\mathcal{H}_{1,n})=n$.
\item $\chia(\mathcal{H}_{2,n})=n+1$.
\item $\chib(\mathcal{H}_{2,n})=n$ for $n\ge2$.
\end{enumerate}
\end{prop}

 \begin{proof} The first assertion is trivial, whereas the second one was already proven by Bartnicki et al. \cite{Bartnicki2008}. Assertion (c) follows straightforwardly from 
Lemma~\ref{helpforKtwo}.~\end{proof}

\section{Modified game based on principal isotopisms}\label{secmodifiedgame}

The most simple case to study the colouring game introduced in the previous section is that based on an feasible principal isotopism, for which the corresponding symbol permutation is the identity.  Recall that, from 
Proposition~\ref{prop_Board}, this is always an extendable isotopism. Let $\Theta=(\pi_1,\ldots,\pi_d,{\rm id})$ be one such an isotopism. Let $P\in\mathcal{PLH}_{\Theta}$ be a configuration of a given 
$\Theta$-stabilized $(a,b)$-colouring game. Rules 1--3 applied to this partial Latin hyper-rectangle $P$ involve that
\begin{itemize}
\item every cell in a given marked orbit of $P$ is empty or coloured with the same colour, and
\item colouring an empty cell in a marked orbit of $P$ does not give us any new restriction on the possible colours for the elements of other cell orbits.
\end{itemize}
Based on both aspects, let us prove that playing the $\Theta$-stabilized $(a,b)$-colouring game on the Hamming graph $\mathcal{H}_{n_1,\ldots,n_d}$ is equivalent to play a modified colouring game on what we call the \emph{orbit contraction graph} ${\cal H}_{n_1,\ldots,n_d}^{\Theta}$. This comes from the contraction of all those vertices in the Hamming graph $\mathcal{H}_{n_1,\ldots,n_d}$ that are related to cells of the same orbit under the action of $\Theta$. The weight $\omega(v)$ of each vertex $v$ coincides with the cardinality of the corresponding cell orbit. The next result involves the existence of a natural neighbourhood relation on the set of cell orbits in $P$ based on that existing in the original Hamming graph ${\cal H}_{n_1,\ldots,n_d}$ and hence, that the orbit contraction graph is well-defined.

\begin{lem}\label{neighborhoodwelldefined} Let $\Theta\in \mathfrak{I}_{n_1,\ldots,n_d,n}$ and let $P\in\mathcal{PLH}_{\Theta}$. There exists a well-defined neighborhood relation on the orbits of $P$ under the action of $\Theta$, which is based on the neighborhood relation of $\mathcal{H}_{n_1,\ldots,n_d}$.
\end{lem}

\begin{proof} Let $\mathfrak{o}_1$ and $\mathfrak{o}_2$ be two distinct orbits of the partial Latin hyper-rectangle $P$ under the action of $\Theta$ and let $v=(i_1,\ldots,i_d)$ and $v'=(i'_1,\ldots,i'_d)$ be two cells in $\mathfrak{o}_1$. There exists a positive integer $m$ such that $\pi^m_k(i_k)=i'_k$, for all $k\leq d$. Then, if there is a cell $w=(j_1,\ldots,j_d)$ in $\mathfrak{o}_2$ such that the edge $vw$ exists in ${\cal H}_{n_1,\ldots,n_d}$, then there is also a cell $w'=(j'_1,\ldots,j'_d)$ in $\mathfrak{o}_2$ such that the edge $v'w'$ exists in ${\cal H}_{n_1,\ldots,n_d}$. Namely, $j'_k=\pi^m_k(j_k)$, for all $k\leq d$.
\end{proof}

For any graph $G=(V,E)$ and every positive integer $l\in\mathbb{N}$, let $G^{\ast l}$ be the vertex-weighted graph having $G$ as base graph and such that every vertex $v\in V$ has weight $\omega(v)=l$. The next result follows then immediately.

\begin{prop}\label{alllcyclesstructure}
Let $\Theta$ be a feasible principal isotopism in $\mathfrak{I}_{n_1,n_2,n}$ with cycle structure $z_{\Theta}=\left(l^{n_1/l},l^{n_2/l},1^n\right)$. Then, the orbit contraction graph ${\cal H}_{n_1,n_2}^{\Theta}$ is isomorphic to the vertex-weighted graph $(\mathcal{H}_{n_1/l,n_2/l}\boxtimes K_l)^{\ast l}$.
\end{prop}

\begin{proof} Suppose $\Theta=(\pi_1,\pi_2,\pi)$. Each pair of cycles from the cycle decomposition of $\pi_1$ and $\pi_2$ determines a vertex of $\mathcal{H}_{n_1/l,n_2/l}$ that 
corresponds in turn to an $(l\times l)$-square in the rectangle associated with $\mathcal{H}_{n_1,n_2}$. There are $l$ adjacent orbits of size $l$ in such a square with regard to the 
adjacency described in Lemma 
\ref{neighborhoodwelldefined}. This square is, therefore, isomorphic to $K_l^{\ast l}$. The adjacency to other orbits is determined by the structure of $\mathcal{H}_{n_1/l,n_2/l}$. Thus, $\mathcal{H}_{n_1,n_2}^{\Theta}$ is isomorphic to $(\mathcal{H}_{n_1/l,n_2/l}\boxtimes K_l)^{\ast l}$.
\end{proof}

Proposition \ref{alllcyclesstructure} cannot be generalized to higher dimensions. To see it, let $d>2$ and let $\Theta=(\pi_1,\ldots,\pi_d,\pi)\in \mathfrak{I}_{n_1,\ldots,n_d,n}$ be 
a feasible principal isotopism with cycle structure $z_{\Theta}=\left(l^{n_1/l},\ldots,l^{n_d/l},1^n\right)$. Similarly to the reasoning exposed in the proof of the mentioned lemma, 
each tuple of cycles from the cycle decomposition of $\pi_1,\ldots, \pi_d$ determines a vertex of $\mathcal{H}_{n_1/l,\ldots,n_d/l}$ that corresponds in turn to an $(l\times 
\ldots\times l)$-hypercube in $\mathcal{H}_{n_1,\ldots,n_d}$, where there are $l^{d-1}$ cell orbits, all of them of length $l$. Each one of these orbits is only adjacent to those cell 
orbits sharing with itself an axis-parallel hyperplane. As a consequence, if $l>2$, then not all the cell orbits of the hypercube are adjacent and hence, this is not isomorphic to 
$K_{l^{d-1}}^{\ast l}$. Nevertheless, if $l=2$, this adjacency holds. In order to deal with this case, let $H_d=\mathcal{H}_{2,\ldots,2}$ be a $(2\times\ldots\times 2)$-hypercube in $\mathcal{H}_{n_1,\ldots,n_d}$. Its vertices can be considered as $d$-dimensional $0-1$-vectors and hence, we can define $H_d^{+{\rm diag}}$ as the graph obtained from $H_d$ by adding diagonal edges connecting each pair of opposite vertices, that is, vertices with coordinates $(x_1,\ldots,x_d)$ and $(1-x_1,\ldots,1-x_d)$. The next lemma follows immediately from this definition of the hypercube $H_d$, which can indeed be done regardless of the dimension $d$.

\begin{lem}\label{lem_Hd} The next results hold.
\begin{enumerate}[a)]
\item $H_1^{+{\rm diag}}\cong K_2$.
\item $H_2^{+{\rm diag}}\cong K_4$.
\item $H_3^{+{\rm diag}}\cong K_{4,4}$.
\end{enumerate}
\end{lem}

The preservation of adjacency that we have previously mentioned enables us to ensure also the next result. From Lemma \ref{lem_Hd}, this is equivalent to Proposition~\ref{alllcyclesstructure} when $d=l=2$.

\begin{prop}\label{mainprophypercube} Let $\Theta$ be a feasible principal isotopism in $\mathfrak{I}_{n_1,\ldots,n_d,n}$ with cycle structure $z_{\Theta}=\left(2^{n_1/2},\ldots,2^{n_d/2},1^n\right)$. Then, the orbit contraction graph ${\cal H}_{n_1,\ldots,n_d}^{\Theta}$ is isomorphic to $(\mathcal{H}_{n_1/2,\ldots,n_d/2}\boxtimes H_{d-1}^{+{\rm diag}})^{\ast 2}$. Particularly, the orbit contraction graph of the hypercube $H_d$ with regard to $\Theta$ is isomorphic to $(H_{d-1}^{+{\rm diag}})^{\ast 2}$.
\end{prop}


\paragraph{Modified coloring game.} We would like to play the $\Theta$-stabilized $(a,b)$-colouring game on the orbit contraction graph ${\cal H}_{n_1,\ldots,n_d}^{\Theta}$ in the same way as on the original Hamming graph. To this end, we have to enable the players the equivalence of colouring an empty cell of a marked cell orbit. We have already exposed that the colour in such a move is already determined and that this kind of move does not give us any new restriction on the game. It can therefore be considered as a passing move. Based on this fact, in our new game we play on the orbit contraction graph by keeping in mind the next two possible moves.
\begin{enumerate}[1.]
\item Colour an uncoloured vertex $v$ of the graph and update $\omega(v)=\omega(v)-1$. This corresponds to colour an empty cell of a symbol-free cell orbit in the original game.
\item Update the weight of a coloured vertex as $\omega(v)=\omega(v)-1$, whenever $\omega(v)>0$. This is a passing move that corresponds to colour an empty cell of a marked cell orbit in the original game.
\end{enumerate}
Alice wins if every vertex of the orbit contraction graph is coloured at the end of the game, otherwise Bob wins. We do not impose any requirement about the final weights of the vertices, because, as we have already exposed, any colouring of the graph ${\cal H}_{n_1,\ldots,n_d}^{\Theta}$ involves in an unique way a colouring of the original Hamming graph. The next result follows straightforwardly from the previous arguments.

\begin{lem}[{\bf Orbit Contraction Lemma}]\label{orbitcontraction}
Alice wins the $\Theta$-stabilized $(a,b)$-colouring game on the Hamming graph ${\cal H}_{n_1,\ldots,n_d}$ if and only if she wins the corresponding modified colouring game on the orbit contraction graph ${\cal H}_{n_1,\ldots,n_d}^{\Theta}$.
\end{lem}

\vspace{0.1cm}

\begin{thm} Let $\Theta=(\pi_1,\pi_2,\pi)$ be an extendable isotopism in $\mathfrak{I}_{n_1,n_2,n}$ with cycle structure $z_{\Theta}=(l,l,1^l)$. Then, $\chi_g^{\Theta}(\mathcal{H}_{n_1,n_2})=l$.
\end{thm}

\begin{proof}
This is a trivial consequence of the orbit contraction lemma, since the base graph of the orbit contraction graph ${\cal H}_{l,l}^{\Theta}$ is isomorphic to $K_l$. Besides, the result only depends on the cycle structure under consideration because of Proposition \ref{prop_CS}.
\end{proof}

\vspace{0.2cm}

Let us finish our study with the discussion of the $\Theta$-stabilized $(a,b)$-game chromatic number of the hypercube $H_d$, for $d\leq 4$. Our motivation to study the game on the hypercube comes from the fact that the orbit contraction graph of the $d$-dimensional hypercube encodes, as we feel, all the complicated structural information of every hyper-rectangle with regard to a principal autotopism where all the cycles in each permutation $\pi_i$ have the same length. Therefore, the game on the hypercube is not only a toy problem, moreover we learn a lot about the game in the higher dimensional case in general.

\begin{thm} Let $a$, $b$ and $d$ be three positive integers and let 
\[\Theta=((12),\ldots,(12),(1)(2))\] 
be a $(d+1)$-dimensional isotopism. The next results hold.
\begin{enumerate}[i)]
\item ${}^{(a,b)}\chi_g^{\Theta}(H_2)={}^{(a,b)}\chi_g(K_2^{\ast2})=2$.
\item ${}^{(a,b)}\chi_g^{\Theta}(H_3)={}^{(a,b)}\chi_g(K_4^{\ast2})=4$.
\item If $d=4$, then
$${}^{(a,b)}\chi_{g_X}^{\Theta}(H_4)={}^{(a,b)}\chi_{g_X}(K_{4,4}^{\ast2})= \left\{\begin{array}{ll}
2,&{\rm if}\ X=A \text{ and } a\geq 2,\\
\min\{b+2,5\},&{\rm if}\ X=A \text{ and } a=1,\\
\min\{b+1,5\},&{\rm if}\ X=B \text{ and } a=1.
\end{array}\right.$$
\end{enumerate}
\end{thm}

\begin{proof}
In the three cases, the first equality follows from Lemma \ref{lem_Hd} and Proposition~\ref{mainprophypercube}, whereas the second one is trivial in $(i)$ and $(ii)$. In $(iii)$, Alice's winning strategy is based on guaranteeing in her first or second move that both bipartite sets are coloured, whereas Bob's winning strategy consists of colouring the vertices of the same bipartite set with different colours. In particular, if $X=A$, then he should select the same bipartite set that Alice has chosen.
\end{proof}

\section{Final remarks and further studies}

In this paper we have introduced a variant of the $(a,b)$-colouring game of the Hamming graph $\mathcal{H}_{n_1,\ldots,n_d}$ for which each position  corresponds to a partial Latin hyper-rectangle having a fixed feasible isotopism $\Theta$ in its autotopism group. We have examined this variant by means of the $\Theta$-stabilized $(a,b)$-game chromatic number, which only depends on the cycle structure of the isotopism $\Theta$. As a first step in the widely spectrum of cases on which this colouring game can be based, several results have been exposed in case of dealing with the bi-dimensional and the hypercube cases. Nevertheless, it is required a deeper study based on the known distribution of isotopism of partial Latin hyper-rectangles according to their cycle structures. The bi-dimensional case, for which such a distribution is known for (partial) Latin squares of small order \cite{Falcon2012, Falcon2013, Stones2012}, is established as an immediate further work.

\begin{prob}
Determine ${}^{(a,b)}\chi_g^{\Theta}$ for $\pi={\rm Id}$ and $d=2$.
\end{prob}

\begin{prob}
Determine ${}^{(a,b)}\chi_g^{\Theta}$ for $\pi={\rm Id}$ and $n=n_1=\ldots=n_d=2$.
\end{prob}

\begin{prob}
Determine ${}^{(a,b)}\chi_g^{\Theta}$ for $\pi={\rm Id}$ and $n=n_1=\ldots=n_d=p$, where $p$ is an odd prime.
\end{prob}

\begin{prob}
Find a unifying description of Hamming graphs w.r.t.\ certain isotopisms.
\end{prob}

Furthermore, our studies motivate to further examine the modified colouring game on arbitrary weighted graphs that are not necessarily based on Latin hyper-rectangles.

\begin{prob}
Determine the maximum $(a,b)$-game chromatic numbers for vertex weighted graphs from interesting classes of graphs (such as forest, outerplanar, planar, or $k$-degenerate graphs etc.) 
assuming a fixed upper bound on the vertex weights.
\end{prob}

\end{document}